\documentclass{article}
\usepackage{amsmath}
\usepackage{amssymb}
\usepackage{amsthm}
\usepackage{amsrefs}
\usepackage{hyperref}

\newcommand{\Cc}{\mathbb{C}}
\newcommand{\Zz}{\mathbb{Z}}

\newcommand{\Rr}{\mathbb{R}}
\newcommand{\N}{\mathbb{N}}
\renewcommand{\H}{\mathcal{H}}

\newcommand{\J}{\mathcal{J}}

\newcommand{\inn}[2]{\langle#1,\,#2\rangle}
\DeclareMathOperator{\ran}{ran}

\DeclareMathOperator{\supp}{supp}

\DeclareMathOperator{\dist}{dist}

\newcommand{\grad}{\nabla}
\newcommand{\Lap}{\Delta}
\newcommand{\di}{\partial}
\newcommand{\cl}{\overline}

\DeclareMathOperator{\diam}{diam}

\newcommand{\s}{\sigma}
\newcommand{\eps}{\epsilon}

\newcommand{\g}{\gamma}
\newcommand{\al}{\alpha}

\renewcommand{\b}{\bar}

\renewcommand{\eqref}[1]{(\textup{\ref{#1}})} 

\theoremstyle{thmstyleone}%
\newtheorem{theorem}{Theorem}
\newtheorem{lemma}{Lemma}

\theoremstyle{thmstyletwo}%
\newtheorem{remark}{Remark}%

\theoremstyle{thmstylethree}%
\newtheorem{definition}{Definition}%

\newcommand{\abs}[1]{\left\lvert#1\right\rvert}
\newcommand{\norm}[1]{\left\lVert#1\right\rVert}
\newcommand{\sbr}[1]{\left[#1\right]}
\newcommand{\Set}[1]{\left\{#1\right\}}
\newcommand{\fullfunction}[5]{\ensuremath{
		\begin{array}{ccrcl}
			{#1}    & \colon  & {#2} & \longrightarrow & {#3} \\
			\mbox{} & \mbox{} & {#4} & \longmapsto     & {#5}
\end{array}}}
\newcommand{\pd}[2]{\ensuremath{\tfrac{\partial#1}{\partial#2}}}
\newcommand{\od}[2]{\ensuremath{\tfrac{d#1}{d#2}}}
\newcommand{\md}[6]{\ensuremath{
		\ifinner
		\tfrac{\partial{^{#2}}#1}{\partial{#3^{#4}}\partial{#5^{#6}}}
		\else
		\tfrac{\partial{^{#2}}#1}{\partial{#3^{#4}}\partial{#5^{#6}}}
		\fi
}}

\newcommand{\del}[1]{\left(#1\right)}
\newcommand{\thmref}[1]{Theorem~\ref{#1}}

\newcommand{\defnref}[1]{Definition~\ref{#1}}

\newcommand{\lemref}[1]{Lemma~\ref{#1}}

\begin{document}
	
	\title{Adiabatic approximation for the motion of Ginzburg-Landau vortex filaments}
	\author{Jingxuan Zhang}
	\maketitle

	\abstract{	
		In this paper, we consider the concentration property of solutions to the dispersive Ginzburg-Landau (or Gross-Pitaevskii) equation in three dimensions.  On a spatial domain, it has long been conjectured that such a solution concentrates near some curve evolving according to the binormal curvature flow, and conversely, that a curve moving this way can be realized in a suitable sense by some solution to the dispersive Ginzburg-Landau equation. Some partial results are known with rather strong symmetry assumptions. 
		
		Our main theorems here provide affirmative answer to both conjectures under certain small curvature assumption. The results are valid for small but fixed material parameter in the equation, in contrast to the general practice to take this parameter to its zero limit. The advantage is that we can retain precise description of the vortex filament structure. The results hold on a long but finite time interval, depending on the curvature assumption.}

	\section{Introduction}
	Consider the dispersive Ginzburg-Landau  (or Gross-Pitaevskii) equation  on a  spatial domain $\Omega\subset \Rr^3$,
	
	\begin{equation}
		i\pd{\psi}{t} =-\Lap \psi+\frac{1}{\epsilon^2}(\abs{\psi}^2-1)\psi.
		\label{1.1}
	\end{equation}
	This equation has its origin from condensed matter physics.
	The complex scalar field $\psi:\Omega\to\Cc$ describes the Bose-Einstein
	condensate in superfluidity, or, ignoring the effect of magnetic
	field, the electronic condensate in superconductivity. 
	Here $\epsilon>0$ is a material parameter measuring the characteristic length scale of  $\abs{\psi}$. 
	
	Realistically, the coherent length
	$\epsilon$ is very small compared to the size of the domain $\Omega$.
	Thus in what follows we take 
	\begin{equation}
		\label{Om}
		\Omega:=\Set{(x,z):x\in\omega\subset\Rr^2, z\in I:=[0,1]},
	\end{equation}
	where $\omega$ is a bounded
	domain in $\Rr^2$, whose size is large compared to the scale $\epsilon$. 
	The choice $I=[0,1]$ can be replaced by any interval $[a,b]$ with $b-a\gg\eps$,
	which amounts to a change of coordinate along $z$-direction. 
	The reason for taking a bounded domain is to 
	circumvent certain undesirable decay properties
	that are not essential to our argument
	(see Sec. \ref{sec:2.1} for details). 
	In addition, for technical reason we also assume $\omega$ is star-shaped 
	around the origin, and the boundary $\di\omega$ is smooth. 
	
	On this cylindrical domain $\Omega$, we impose the boundary conditions 
	$\abs{\psi}\to 1 $ as $x$ approaches $\di\omega$ horizontally, and $\psi(x,0)=\psi(x,1)$
	for every $x\in \omega$.  
	Note that through the transform $\psi=e^{i\epsilon^{-2}t}u$, equation \eqref{1.1} 
	becomes the cubic defocusing subcritical nonlinear Schr\"odinger equation.
	Global well-posedness for such equations  has been known since the eighties \cite{MR533218}.
	Using  standard blow-up arguments, one can show that 
	for every initial configuration $u_0\in H^1(\Omega)$ satisfying
	the boundary conditions above, there exists 
	a unique global solution $u_t\in H^1(\Omega)$ to \eqref{1.1}
	generated by $u_0$. 
	Therefore in this paper we will mainly work with
	Sobolev spaces $H^k,\,k\ge1$.
	
	\subsection{The geometric structure of \eqref{1.1}}
	\label{sec:1.1}
	
	The dispersive Ginzburg-Landau equation \eqref{1.1} can be viewed as a Hamiltonian system as follows.
	Let   $X\subset L^2(\Omega,\Cc)$ be a suitable configuration space for \eqref{1.1}.
	If we identify the space $L^2(\Omega,\Cc)=L^2(\Omega,\Rr)\times  L^2(\Omega,\Rr)$
	through $u\mapsto (\Re u,\Im u)$, then we can regard $X$ as 
	a real vector space,
	equipped with
	the inner product $\inn{(\Re u,\Im u)}{(\Re v,\Im v)}=\int(\Re u\Re v+\Im u\Im v).$
	
	Under this identification, the operator 
	$J:\psi\mapsto -i\psi$ can be represented by the symplectic matrix
	\begin{equation}
		\label{4.0}
		J=\begin{pmatrix}0&1\\-1&0\end{pmatrix}.
	\end{equation}
	Thus we say $J$ is a symplectic operator, in the sense that 
	it induces a symplectic form $\tau=\inn{J\cdot}{\cdot}$ on $X$, satisfying
	$\tau({w},{w'})=-\tau({w'},{w})$ for every $w,w'\in X$. 
	
	Using the operator $J$, we can write \eqref{1.1} as the Hamiltonian system
	\begin{equation}
		\label{1.2}
		\pd{\psi}{t}=JE'(\psi),\quad E(\psi)=E_\Omega^\epsilon(\psi):=\int_\Omega \frac{1}{2}\abs{\grad\psi}^2+\frac{1}{4\epsilon^2}\del{\abs{\psi}^2-1}^2.
	\end{equation}
	Here $E$ is the Ginzburg-Landau energy, measuring the difference 
	of Helmholtz free energy between a transition phase $\psi$ and the normal phase.
	$E'(\psi)$ denotes the $L^2$-gradient of the energy functional w.r.t.  the real inner product given above. 
	This gradient $E'(\psi)$ is given explicitly as the r.h.s. of \eqref{1.1}.  
	
	It is well-known by now that as the material parameter $\epsilon\to0$ in
	\eqref{1.1}, the energy $E_\Omega^\epsilon$ of a
	finite-energy configuration $\psi$ 
	concentrates  near  some lower dimensional
	manifold $\gamma$. 
	This
	concentration phenomenon can be made precise
	in measure theoretic terms. For instance, see \cite{MR1491752,MR1864830,MR2195132}. 
	This leads one to expect, at least at the heuristic level,
	that as $\epsilon\to0$, \eqref{1.2} reduces, in an appropriate manner, to
	a Hamiltonian system in the space of curves, 
	\begin{equation}
		\label{1.3}
		\pd{\gamma}{t}=\J L'(\gamma),
	\end{equation}
	where $\gamma=\gamma_t$ is a $C^1$ path of curves in $\Omega$,
	and $\J$ is a suitable (symplectic) operator.

	\subsection{Outline of the main results}
	Suppose $\gamma$ in \eqref{1.3} is parametrized by arclength, $L(\gamma)=\int\abs{\di_s\gamma}^2$, $L'(\gamma)=-\di_{ss}\gamma$,  and $\J:=-\di_s\gamma\times$.
	Then \eqref{1.3} becomes the binormal curvature flow \eqref{BNF}. 
	The purpose of this paper is to formulate rigorously the connection
	between the dispersive Ginzburg-Landau equation \eqref{1.1} and the binormal curvature flow.
	
	To this end, we first make precise the notion of
	concentration set in \lemref{lem2.2} 	for a class of low energy configurations.
	Then our first main result, \thmref{thm3.1}, states that if $u_t$ is a solution
	to \eqref{1.1} generated by $u_0$, and $u_0$ has concentration set $\gamma_0$,
	then the flow $u_t$ has concentration set given in the leading order by $\gamma_t$, the flow generated by $\gamma_0$ under the binormal
	curvature flow \eqref{BNF}.
	As a corollary, we derive the second main result, \thmref{thm3.2}, which states the converse: If $\gamma_t$ is the flow generated by $\gamma_0$,
	then we can find some solution $u_t$ to \eqref{1.1} such that the concentration
	set associated to $u_t$ is given in the leading order by $\gamma_t$. 
	Both main results hold on some long but possibly finite time interval
	depending on the initial configuration.
	
	The precise statements of the main results of this paper are as follows:
	\begin{theorem}
		\label{thm0}
		\begin{enumerate}
			\item For any $\epsilon>0$, there are $\delta_1,\,\delta_2\ll\epsilon$ such that the following holds:
			Let $M$ be the manifold of approximate vortex filaments as in \defnref{M}.
			Let $u_0\in X^1+\psi_0$ be an initial configuration in
			the energy space ($X^k$ and  $\psi_0$ are defined in Sec. \ref{sec:2.2}), 
			such that $\dist_{X^2}(u_0, M)<\delta_2$.
			Let $u_t$ be the flow generated by $u_0$ under \eqref{1.1}.

			Then there is some $T>0$ independent of $\epsilon$ and $\delta_1,\,\delta_2$, 
			such that for $\epsilon t\le T$, there exists a flow of moduli $\s_t$ asscociated to $u_t$ as in \lemref{lem2.2}, with
			\begin{equation}
				\norm{u_t-f(\s_t)}_{X^2}=o(\sqrt{\delta_1}),
			\end{equation}
			where $f:\Sigma_{\delta_1}\to M$ is the parametrization defined in \eqref{2.1}.
			
			Moreover,  for $\epsilon t\le T$,  the flow of moduli
			moduli $\s_t$ evolves according 
			to 
			\begin{equation}
				\di_t\sigma=\J_\sigma^{-1}d_\sigma E(f(\sigma))+o_{\norm{\cdot}_{Y^0}}(\delta_1),
			\end{equation}
			where the operator $\J_\s$ is defined in \eqref{2.3}.
			
			\item	For any $\beta>0$, there exist $\delta,\epsilon_0>0$ such that the following holds: Let $\epsilon<\epsilon_0$ in \eqref{1.1}. 
			Let $\sigma_0=(\lambda_0,\g_0)\in\Sigma_\delta$ be given,
			where $\Sigma_\delta$ is defined in \eqref{Si}.
			Let $\vec\g_t$ be
			the flow generated by $(\gamma_0(z),z)$ under the binormal curvature flow \eqref{BNF}.
			
			Then there exist a solution $u_t$ to \eqref{1.1},
			and  some $T>0$ independent of $\epsilon$ and $\delta$, 
			such that for all $\epsilon t\le T$, $X\in C^1_c(\Rr^3,\Rr^3)$ and $\phi\in C^1_c(\Rr^3,\Rr)$ with $\norm{X}_{C^1},\,\norm{\phi}_{C^1}=O(\delta^{-1/4})$, we have
			\begin{align}
				&\abs{\int_\Omega X\times Ju-\pi\int_{\vec\g_t}X}\le \beta,\\
				&\abs{\int_\Omega \frac{e(u)}{\abs{\log\epsilon}}\phi-\pi\int_{\vec\g_t}\phi\,d \H^1}\le \beta  ,
			\end{align}
			where for $\psi=\psi^1+i\psi^2$,
			$$
			e(\psi)=	\frac{1}{2}\abs{\grad\psi}^2+\frac{1}{4\epsilon^2}(\abs{\psi}^2-1)^2,\quad  J\psi=\grad\psi^1\times \grad\psi^2.$$
		\end{enumerate}
	\end{theorem}
	
	\begin{remark}
		In the formulation of \thmref{thm0}, we do not adopt the convention
		of taking $\epsilon\to0$. Instead, for \textit{small but fixed }$\epsilon$, the main assumption is that  the concentration sets
		$\gamma_t,\,t\ge0$ have uniformly small curvature. (See the definitions
		of the various spaces in \thmref{thm0} in Sec. \ref{sec:2.2}.)
		This allows us to retain precise description of the vortex structure of the
		evolving configurations.
	\end{remark}
	
	We obtain the  results in \thmref{thm0} by developing an adiabatic 
	approximation scheme for \eqref{1.1}. This allows us to decompose a flow $u_t$ solving \eqref{1.1} to a slowly
	evolving main part $v_t$
	and a uniformly small remainder $w_t:=u_t-v_t$. The flow $v_t$ consists of some low energy configurations
	whose motion is explicitly governed by their concentration sets $\gamma_t$.
	The small curvature assumption on $\gamma_t$ is used here to ensure $v_t$ has low energy,
	which is essential for the validity of adiabatic approximation. (If the curvature is large at a point on the concentration curve, then the speed of the vortex filament at this point is large, and consequently the adiabatic approximation breaks down.)
	
	The main technical ingredients include 1. to show that the concentration sets of $v_t$ indeed evolve according to the binormal curvature flow \eqref{BNF} to the leading order, \textit{assuming the the remainder $w_t$ is small}
	and 2. to get apriori estimate on the remainder  $w_t$. The second part is analogous
	to proving dynamical stability of solitons, except for that the configurations $v_t$ need not to
	be stationary.

	\subsection{Historical remark} 
	Let us briefly comment on the existing literature on related problems.
	If the evolution \eqref{1.1} is instead the \textit{energy dissipative} dynamics (so that \eqref{1.1} becomes a cubic nonlinear heat equation), rigorous result 
	that formulates the kind of connection in which we are interested is known \cite{MR2195132}. This is one example of a vast literature in the $\epsilon\to0$ scheme. 
	See an excellent review \cite{MR3729052} of works along this line.
	Note our method is completely different from the measure theoretic schemes for the $\epsilon\to0$ limit.
	One can also ask
	similar questions about the hyperbolic and gauged analogues of \eqref{1.1}. The method we develop in the present paper is applicable to the latter cases,
	which we will treat elsewhere.

	The problem we study here is motivated by a conjecture of R.L. Jerrard \cite[Conj. 7.1]{MR3729052}. The method we use here is inspired by
	a series of papers by I.M. Sigal with several co-authors \cite{MR1644389,
		MR1644393,
		MR2189216,MR2465296,
		MR3803553,
		MR3824945}.
	Other results along this line which we have referred to include \cite{MR1257242,MR1309545,MR2097576,MR2525632,MR2576380,MR2805465,MR3397388,MR3645729,MR3619875,MR4107807}.
	Most of these papers consider planar problems (possibly after some symmetry reduction),
	or 
	higher dimensional problems with
	concentration sets given by finite collections of  points. In this regard,
	the geometry of our problem is more involved.

	\subsection{Arrangement}
	The structure of this paper is as follows: In Section 2, we construct the
	manifold $M$ consisting of what we call approximate vortex filaments.
	These are low energy configurations with explicit concentration sets.
	We find a path of  appropriate  approximate filaments $v_t$
	as the ``adiabatic part'' of a given 
	solution $u_t$ to \eqref{1.1}, provided $u_t$ stays uniformly close to $M$. 
	In Section 3, we show that the path $v_t$ as above
	is governed by the binormal curvature flow \eqref{BNF}. 
	We prove then two main theorems by finding an apriori estimate for  $w_t:=u_t-v_t$,
	valid on a long but finite interval.
	The proofs rely on the adiabaticity of the approximate filaments, as well as the energy conservation for \eqref{1.1}. In Section 4, we discuss the properties of certain linearized operators involved in the preceding analysis.
	In Appendix, we recall some basic
	concepts of Fr\'echet derivatives that we use repeatedly, and collect various
	technical estimates for  operators used earlier.

	\subsection*{Notations} Throughout the paper, when no confusion arises, we shall drop the time dependence $t$ in subscripts.
	A domain $\Omega\subset\Rr^d$ is an open connected subset.
	The symbols $L^p(\Omega),\,H^s(\Omega),\, 0<p\le \infty,\,s\in \Rr$ denote respectively the Lebesgue space of order $p$, and
	the Sobolev space of order $s$, consisting of functions from a domain
	$\Omega$ into $\Cc$.

	\section{Approximate vortex filaments}
	In this section, we construct the manifold of approximate (vortex) filaments
	and discuss their key properties.

	\subsection{The planar 1-vortex}\label{sec:2.1}
	On the entire plane $\Rr^2$, it is well-known  that \eqref{1.1} has smooth, stationary, radially symmetric solutions
	$\psi^{(n)}:\Rr^2\to\Cc$, where $n\in\Zz$ labels the winding number
	$$\deg\psi\vert_{\abs{x}=R}=n$$ for large $R\gg0$. 
	The characteristic feature of $\psi^{(n)}$ is that they 
	concentrate near the origin $r=0$.
	Among these vortices,  the simple ones with $\abs{n}=1$ are stable, and
	the higher order ones with $\abs{n}>1$ are unstable \cite{MR1479248}.
	
	For this reason, in what follows  we will only be concerned with
	the $1$-vortex.
	
	Write $\psi^{(1)}=\varphi(r)e^{i\theta}$ in polar coordinate. Then $\varphi\in C^\infty$, $\varphi'>0$
	for $r>0$, and the following 
	asymptotics hold \cite[Sec. 3.1]{MR1763040}:
	\begin{equation}
		\label{2.0}
		\varphi\sim 1-\frac{\epsilon^2}{2r^2}\quad(r\to\infty),\quad \varphi\sim \frac{r}{\epsilon}-\frac{r^3}{8\epsilon^3}\quad (r\to0).
	\end{equation}
	
	The planar stationary equation
	$$-\Lap \psi+\frac{1}{\epsilon^2}\del{\abs{\psi}^2-1}\psi=0, \quad \psi:\Rr^2\to\Cc$$
	is called the  the (ungauged) Ginzburg-Landau equation. It has rotation, translation,
	and global gauge symmetries. Among these the, latter two classes
	are broken by $\psi^{(1)}$. 
	Consequently, if  $L^{(1)}_x:H^2(\omega)\to L^2(\omega)$
	is the linearized operator at $\psi^{(1)}$,
	then the vectors 
	$$\pd{\psi^{(1)}}{x_j},\,i\psi^{(1)},\quad j=1,2$$
	are in the kernel of $L^{(1)}_x$. We call these vectors \textit{the symmetry zero modes.}
	Note that these modes  
	are not in $L^2(\Rr^2)$.

	On a bounded star-shaped domain $\omega\subset \Rr^2$ around the origin, the following estimates are well-known
	\cite[Chap. III, X]{MR1269538}:
	\begin{align}
		&E_\omega(\psi^{(1)})\le\pi\abs{\log\epsilon}+C(\omega)\label{2.0.1},\\
		&\norm{\psi^{(1)}}_{L^\infty(\omega)}\le1\label{2.0.4},\\
		&\norm{\grad\psi^{(1)}}_{L^\infty(\omega)}\le\frac{C(\omega)}{\epsilon}\label{2.0.2},\\	&\norm{\grad\psi^{(1)}}_{L^2(\omega)}\le C(\omega)\abs{\log\epsilon}^{1/2}\label{2.0.5},\\
		&\frac{1}{\epsilon^2}\int_\omega \del{\abs{\psi^{(1)}}^2-1}^2\le C(\omega)\label{2.0.3}.
	\end{align}
	The finite energy property \eqref{2.0.1} fails if $\omega$ is unbounded.
	Consequently, if  $\omega$ is unbounded, then in general (say $\omega=\Rr^2$), even the simple
	planar vortex $\psi^{(1)}$ has infinite
	energy. Moreover, the translation zero modes are not $L^2$.
	
	One can study the problem with $\omega=\Rr^2$ by
	either using some kind of renormalized energy \cite{MR1479248}, or
	posing the problem in a weighted Sobolev space \cite{MR1763040}.
	Though  it will incur significantly more involved estimates,
	we believe the arguments below can extend to the non-compact case using one of these methods.
	
	\subsection{Construction of approximate vortex filaments}\label{sec:2.2}

	To construct the manifold of approximate vortex filaments,
	we first define some configuration spaces.
	Put 
	\begin{align}
		\label{X}
		&X^s:=\Set{\psi\in H^s(\Omega,\Cc):\psi\vert_{\di\omega\times I}=0,\;\psi(x,0)=\psi(x,1)\text{ for every } x\in \omega},\\
		&Y^k:=\Rr\times C^k(I,\Rr^2)\quad (s\in\Rr,\, k\in \N).\label{Y}
	\end{align}
	We write elements in $Y^k$ as $\sigma=(\lambda,\gamma)$. 
	$X^s,\,Y^k$  are real Hilbert spaces with the inner products given
	respectively by
	\begin{align}
		\label{2.00X}
		&\inn{\psi}{\psi'}_X=\int_\Omega \Re(\cl{\psi}\psi')\quad (\psi,\psi'\in X^s),\\
		&\inn{\sigma}{\sigma'}_Y=\int_I \gamma\cdot \gamma'+\mu\mu'\quad(\s,\s'\in Y^k).\label{2.00Y}
	\end{align}
	The norm on $Y^k$ is $\norm{(\lambda,\g)}_{Y^k}:=\norm{\gamma}_{C^k}+\abs{\lambda}$.
	The  Ginzburg-Landau energy $E$ defined in \eqref{1.2} is smooth
	on $1+X^k$ with $k\ge1$, which we call \textit{the energy space}. 
	
	Write $$C^k_\text{per}:=\Set{\gamma\in C^k(I,\omega):\gamma(0)=\gamma(1)}\quad (k\in \N).$$
	Here we require periodic boundary condition for $\gamma$, so as to match 
	the periodic boundary condition along the
	$z$-axis for \eqref{1.1}. 
	Indeed, we impose such boundary conditions 
	so that the arguments below can be naturally extended to an infinitely long cylindrical 
	domain with $I=\Rr$. In general, one can take $I=[0,a]$ for any $a\gg\eps$,
	and impose other appropriate boundary conditions. 
	Notice that if one varies the vertical boundary condition  for \eqref{1.1},
	then the definition of $C^k_\text{per}$ must be changed accordingly.
	
	Define 
	\begin{align}
		&\Sigma:=\Rr\times \Set{\gamma \in C^2_\text{per}:z\mapsto (\gamma(z),z)\text{ is an embedding of $I$ into $\Omega$}},\\ 
		&\Sigma_\delta:=\Set{\s\in \Sigma:\norm{\s}_{Y^2}<\delta}.\label{Si}
	\end{align}
	We view $\Sigma$ as a manifold.
	Each fibre $T_\s\Sigma$ can be trivialized as a subspace of $Y^0$. 
	We view 
	$\Sigma_\delta$ as an open submanifold of $\Sigma$.

	\begin{definition}[Manifold of approximate vortex filaments]
		\label{M}
		For  $\g\in C^k_\text{per}$ and a function $\psi:\Omega\to \Cc$,
		let $\psi_\g(x,z):=\psi^{(1)}(x-\gamma(z))$, where
		$\psi^{(1)}$ is the simple planar vortex in Sec. \ref{sec:2.1}.

		Define a map 
		\begin{equation}
			\label{2.1}
			\fullfunction{f}{\Sigma_\delta}{X^0+\psi_0}{\sigma=(\lambda,\gamma)}{e^{i\lambda}\psi_\g(x,z)},
		\end{equation}
		where $\psi_0:\Omega\to \Cc$ is the lift of the planar vortex $\psi^{(1)}$
		to $X^0$.
		The map $f$ is $C^1$, as we show in Appendix. $f$ parametrizes a submanifold $$M:=f(\Sigma_\delta)\subset X^0+\psi_0.$$ 
		The tangent space to $M$ at $f(\sigma)$ is $T_{f(\sigma)}M=df(\sigma)(T_\sigma \Sigma_\delta)$, where $df(\sigma)$ denotes the Fr\'echet derivative of $f$ at $\s$, given explicitly 
		in Appendix. 
		We call the elements in $M$ the \textit{approximate vortex filaments}.
	\end{definition}
	\begin{remark}
		The construction of $M$ is motivated by the broken symmetries by $\psi^{(1)}$.
		We use the term  
		\emph{approximate (vortex) filaments}
		because the configurations in $M$ concentrate near some curves around 
		$\Set{0}\times I$
		by construction.
		We can trivialize   $T_{f(\s)}M $ as a subspace of $X^0$.
		The spaces $M,\,\Sigma$ are Riemannian manifolds
		w.r.t.  the inner products given in \eqref{2.00X}-\eqref{2.00Y}. 
	\end{remark}
	
	\subsection{Properties of approximate filaments}
	In what follows, we always assume the material parameter $\epsilon\ll1$ in \eqref{1.1}.
	
	A key observation is that since $u:=\abs{\psi^{(1)}}$ (resp. $v:=\abs{\grad\psi^{(1)}}$) is strictly increasing (resp. decreasing)
	sufficiently away from $r=0$, using  the asymptotics in \eqref{2.0}, we have 
	control over the oscillation of $u$ and $v$ as
	\begin{equation}
		\label{4.1}
		\abs{u(r)-u(s)}\le C\frac{\epsilon^2}{R^2},\quad   
		\abs{v(r)-v(s)}\le C\frac{\epsilon}{R}\quad(r>s\ge R\gg0),
	\end{equation}
	where $C$ is independent of $\epsilon$.
	
	Let $\alpha>0$ be given such that the planar domain $\omega$ contains the ball of radius $1+\epsilon^\alpha$. Then it is not hard to see that for $\s\in\Sigma_{\epsilon^\alpha}$,
	$$\begin{aligned}
		\norm{f(\s)}_{X^0}^2&=\int_\Omega\abs{\psi_\g}^2\\
		&=\int_I\int_\omega \abs{\psi_\g}^2(x,z)\,dxdz\\
		&\le\int_I\del{\int_\omega \abs{\psi^{(1)}}^2(x)\,dx+\int_\omega \del{\abs{\psi_\g}^2(x,z)-\abs{\psi^{(1)}}^2(x)}\,dx}\,dz\\
		&\le\int_I\del{\int_\omega \abs{\psi^{(1)}}^2(x)\,dx+\norm{\s}_{Y^0}\diam(\omega)\sup_{r>s\ge1}\abs{u(r)-u(s)}^2}\,dz\\
		&\le\norm{\psi^{(1)}}_{L^2(\omega)}^2+C(\omega)\epsilon^{4+\alpha}.
	\end{aligned}$$
	Therefore we have 
	\begin{equation}
		\label{4.2}
		\norm{\psi_\g}_{X^0}=\norm{\psi^{(1)}}_{L^2(\omega)}+O(\epsilon^{2+\alpha/2}).\end{equation}
	Similarly, 
	one can show using \eqref{4.1} that 
	\begin{equation}
		\label{4.3}
		\norm{\grad_x\psi_\g}_{X^0}=\norm{\grad_x\psi^{(1)}}_{L^2(\omega)}+O(\epsilon^{1+\alpha/2}).
	\end{equation}
	
	
	In what follows, for a $C^1$-curve $\gamma:I\to \Rr^d$, we write
	$\gamma_z=\di_z\gamma$, etc.. (Not to be confused with the 
	subscript $t$ in  time-parametrized families.)
	
	\begin{lemma}[approximate critical point]
		\label{lem2.1}
		Let $\alpha>0$.
		If $\sigma\in \Sigma_{\epsilon^\alpha}$, then $\norm{E'(f(\sigma))}_{X^0}\le C\epsilon^{\alpha}\abs{\log\epsilon}^{1/2}$ where $C$ is independent
		of $\s$. 
	\end{lemma}
	\begin{proof}
		Using the fact that $\psi^{(1)}$ is stationary solution to \eqref{1.1}, one can
		compute
		$$E'(f(\sigma))=e^{i\lambda}\del{\grad_x \psi_\g \cdot \gamma_{zz}-\grad^2_x\psi_\g\gamma_z\cdot \gamma_z}.$$
		For $\norm{\gamma}_{C^2}\ll1$, the leading order term in this
		expression is 
		$\grad_x \psi_\g \cdot \gamma_{zz}$. One can estimate this as 
		$$\begin{aligned}
			\norm{\grad_x \psi_\g \cdot \gamma_{zz}}_{X^0}&\le2\norm{\grad_x\psi_\g}_{X^0}\norm{\gamma_{zz}}_{C^0}\\
			&\le 2 \del{\norm{\grad_x\psi^{(1)}}_{L^2(\omega)}+C(\omega)\epsilon^{1+\alpha/2}}\norm{\gamma}_{C^2}\\&\le C(\omega)\epsilon^{\alpha}\abs{\log\epsilon}^{1/2},
		\end{aligned}$$
		where in the last step one uses $\s\in\Sigma_{\epsilon^\alpha}$ and the estimates \eqref{2.0.5}, \eqref{4.3}.
	\end{proof}

	To simplify notation, write $g_\s:Y^0\to X^0$ for the action of $df(\s):T_\s\Sigma_\delta\to T_{f(\s)}M$ on each fibre. 
	Let $g_\s^*$
	be the adjoint to $g_\s$ w.r.t.  the inner products defined in \eqref{2.00X}-\eqref{2.00Y}.
	In Appendix, we calculate these operators explicitly in \eqref{A.1}-\eqref{A.2}. 
	
	\begin{remark}
		
		Note here that 1. $g_\sigma$ is 
		injective, and therefore $f$ is an immersion;  2. Using the 
		regularity of $\psi^{(1)}$ and Sobolev embedding, we
		have
		$$g_\s:Y^0\to X^s,\quad g_\s^*:X^r\to Y^0\quad (s\in\Rr,r\ge2).$$
	\end{remark}
	
	Let $J: X^0\to X^0$ 
	be the symplectic operator sending $\psi$ to $-i\psi$. 
	We show in Appendix that the map
	\begin{equation}
		\label{2.3}
		\fullfunction{\J_\sigma}{Y^0}{Y^0}{\xi}{g_\s^*J^{-1}g_\s\xi}
	\end{equation}
	defines a symplectic operator, in the sense that $J_\s$
	induces a symplectic form on the tangent bundle $T\Sigma$.
	Moreover,  $\J_\sigma$ is invertible, 
	and satisfies the 
	uniform estimate $\norm{\J_\sigma}_{Y^k\to Y^k}\le C(\Omega)\abs{\log\eps}^{-2}$ for any $k\in\N$.
	
	Geometrically, notice that $\J_\s$ is the pullback of $J$ by the parametrization $f$.
	Recall in Sec. \ref{sec:1.1}, we have defined the bilinear  map
	induced by $J$ as
	\begin{equation}
		\label{tau}
		\tau:(u,v)\mapsto \inn{Ju}{v}_X\quad (u,v\in X^s).
	\end{equation}
	This $\tau$ is a non-degenerate symplectic form on $X^s$. 
	Thus through the immersion $f$, the manifold of approximate filaments
	$M$ also inherits a symplectic structure, with 
	a non-degenerate symplectic form induced by $\J_\s$.
	
	\subsection{The adiabatic decomposition}
	\label{sec:2.4}
	
	In this section we derive a key result that is essential for the 
	development in Section 3. 
	The point is that 
	on a tubular neighbourhood \emph{with definite volume} around the manifold $M$ of approximate filaments, one can define a nonlinear 
	projection into $\Sigma$. This way we can make precise 
	the notion of concentration set for low energy configurations that are
	uniformly close to the approximate filaments. 
	
	Recall the symplectic form $\tau$ is given in \eqref{tau}.
	
	\begin{lemma}[concentration set]
		\label{lem2.2} Let  $M:=f(\Sigma_{\eps^\alpha})$ be the manifold of
		approximate vortex filaments.
		Then there is $\delta>0$ such that for every $u\in X^1+\psi_0$ with $\dist_{X^2}(u,M)<\delta$,
		there exists a unique $\sigma\in \Sigma_{\epsilon^\alpha}$ such that 
		\begin{equation}
			\label{2.5}	\tau(u-f(\sigma),\phi)=0 \quad (\phi\in T_{f(\sigma)}M). 
		\end{equation}
		
	\end{lemma}
	\begin{remark}
		\label{concentrationset}
		We call  $\sigma=(\lambda,\gamma)$ associated to $u$ the \textit{moduli} of the
		latter. This terminology is taken from e.g. \cite{MR1257242,MR1309545} with  similar context. 
		
		We call the curve $\gamma$ \textit{the concentration
			set} of $u$. The real number $\lambda$ is the global gauge parameter and is not physically meaningful.
		One can view \eqref{2.5} as an orthogonality condition w.r.t. the symplectic form \eqref{tau},
		and the association $u\mapsto \s$ is optimal in this sense.
	\end{remark}
	\begin{proof}
		1. First, for $\sigma\in\Sigma_{\epsilon^\alpha}$, define the linear projection $Q_\sigma$ by 
		\begin{equation}
			\label{2.4}
			\fullfunction{Q_\sigma}{X^0}{ X^0}{\phi}{g_\sigma \J_\sigma^{-1}g_\sigma^*J^{-1}\phi}.
		\end{equation}
		This map $Q_\s$ is the skewed (i.e. $Q_\s^*=J^*Q_\s J$) projection onto
		the tangent space $T_{f(\s)}M$. 
		In the definition of $Q_\sigma$, each factor is uniformly bounded, as we show in Appendix \eqref{A.1}-\eqref{A.2} and \eqref{A.6}.
		So we get the uniform estimate $\norm{Q_\sigma}_{X^0\to X^0} \le C.$
		One can check $Q_\sigma\phi=\phi$ for every $\phi\in T_{f(\sigma)}M$ by
		writing $\phi=g_\sigma\xi$ for some $\xi\in Y^0$, since 
		$Q_\s g_\s\xi = g_\s\J_\s^{-1}(g_\s^*J^{-1}g_\s)\phi=g_\s(\J_\s^{-1}\J_\s)\xi=g_\s\xi=\phi$.

		Next, we find the concentration set $\sigma$ using Implicit Function Theorem.
		Consider the map 
		$$\fullfunction{F}{X^2\times \Sigma_{\epsilon^\alpha}}{Y^0\times \Rr}{(\phi,\sigma)}{g_\sigma^*J^{-1}(\phi-f(\sigma))}.$$
		Condition \eqref{2.5} is satisfied if $Q_\sigma (u-f(\s))=0$.
		To see this, one uses the property $Q^*=-JQJ$,
		which implies 
		$\tau(Q\phi,\phi')=\tau(\phi,Q\phi')$.
		Thus, if $F(\phi,\sigma)=0$, then \eqref{2.5}
		is satisfied.
		
		By construction, we have the following expression
		for the  partial Fr\'echet derivative 
		$$\di_\sigma F\vert_{(f(\sigma),\sigma)}=-\J_\sigma.$$ 
		It is invertible since $\J_\sigma$ is invertible, see Appendix.
		The equation $F(\phi,\sigma)=0$ has the trivial solution
		$(f(\sigma),\sigma)$. It follows that for any fixed $\sigma\in\Sigma_\epsilon$, there is $\delta=\delta(\sigma,\epsilon)>0$
		and a map $S_\sigma:B_\delta(f(\sigma))\to \Sigma_\epsilon$ such that 
		$F(\phi,S(\phi))=0$ for $\phi\in B_\delta(f(\sigma))$.
		
		2. It remains to show that in fact $\delta$ can be made independent of $\sigma$.
		This is important because we must retain a definite volume for the 
		projeciton neighbourhood, so that later on a flow can fluctuate within
		this neighbourhood.
		
		Write 
		$$A_\phi:=\di_\sigma F\vert_{(\phi+f(\sigma),\sigma)},\quad V_\phi:=A_\phi-A_0.$$
		Then $$A_0=-\J_\sigma,\quad V_\phi=(d_\sigma g_\sigma^*)\del{\cdot}\vert_{J^{-1}\phi}.$$
		The size of $\delta$ is determined by the condition 
		that for every $\phi\in B_\delta(f(\sigma))$,
		\begin{align}
			&\label{2.4.00}
			A_\phi\text{ is invertible},\\
			&\label{2.4.0}\norm{A_\phi^{-1}}_{Y^0\to Y^0}\le \frac{1}{4\norm{\J_\sigma^{-1}}_{Y^0\to Y^0}},\\
			&\label{2.4.1}
			\norm{F(\phi+f(\sigma),\sigma)}_{Y^0}\le \frac{\delta_0}{4\norm{\J_\sigma^{-1}}_{Y^0\to Y^0}},
		\end{align}
		where $\delta_0$ is determined 
		that for every $\xi\in B_{\delta_0}(\sigma)$ and $\phi\in B_\delta(f(\sigma))$,
		\begin{align}
			&\norm{R(\phi,\xi)}_{Y^0}\le  \frac{\delta_0}{4\norm{\J_\sigma^{-1}}_{Y^0\to Y^0}},\label{2.4.2}\\ &R(\phi,\xi):=F(\phi+f(\sigma),\sigma+\xi)-F(\phi+f(\sigma),\sigma)-\di_\sigma F(\phi+f(\sigma),\sigma)\xi,\notag\\
			&\norm{\J_{\sigma+\xi}-\J_\sigma}_{Y^0\to Y^0}\le\frac{1}{4\norm{\J_\sigma^{-1}}_{Y^0\to Y^0}}.\label{2.4.3}
		\end{align}
		See for instance \cite[Sec. 2]{MR1336591}. 
		Note that the r.h.s. of \eqref{2.4.0}-\eqref{2.4.3} are independent of $\sigma$ 
		by the uniform estimate for $\J_\sigma^{-1}$.

		Conditions \eqref{2.4.2}-\eqref{2.4.3} are satisfied for some 
		$\delta_0=C(\Omega)\epsilon$.  To get \eqref{2.4.2}, one uses \eqref{A.6} and the fact that $\norm{R(\phi,\xi)}_{Y^0}=o(\norm{\xi}_{Y^0})$, since it
		is the super-linear remainder of the expansion of $F$ in $\sigma$.
		To get \eqref{2.4.3}, one uses the continuity of the map
		$\sigma \mapsto \J_\sigma \in L(Y^0,Y^0).$  
		
		The claim now is that \eqref{2.4.00}-\eqref{2.4.1} are satisfied for  $\delta=O(\epsilon^3)$.
		Indeed, plugging $\delta_0=C\epsilon$ to \eqref{2.4.1} and using the uniform estimate for 
		$g_\sigma^*$ and $\J_\sigma^{-1}$, one sees that \eqref{2.4.1} is satisfied so long
		as $\delta=O(\epsilon\abs{\log\epsilon}^{3/2}).$

		By elementary perturbation theory, since 
		$A_0$ is invertible, and the partial Fr\'echet derivative $A_\phi$ 
		is continuous in $\phi$ as a map from $X^2\to L^2(Y^0, Y^0)$, it follows that condition \eqref{2.4.00} is satisfied provided
		$\norm{V_\phi}_{Y^0\to Y^0}<\norm{A_0^{-1}}_{Y^0\to Y^0}^{-1}\le C\abs{\log\epsilon}^2.$
		By the uniform estimate on $d_\sigma g_\sigma^*$, we can arrange this with $\delta=O(\epsilon^3)$.
		
		Lastly, referring to the Neumann series for the inverse 
		$$A_\phi^{-1}=\sum_{n=0}^\infty A_0^{-1}\del{-V_\phi A_0^{-1}}^n,$$ 
		one can see that $A_\phi^{-1}$ is also continuous in $\phi$, and  $\norm{A_\phi^{-1}}_{Y^0\to Y^0}=O(\abs{\log\epsilon}^{-2})$
		so long as $\norm{V_\phi}_{Y^0\to Y^0}=o(\abs{\log\epsilon}^2)$.
		The latter holds with $\delta=O(\epsilon^3)$. This completes the proof.
	\end{proof} 
	
	\lemref{lem2.1} gives a unique decomposition for every configuration $u$ sufficiently
	close to $M$ as $u=v+w$, where $v$ is in $M$,
	and $\tau(v,w)=0$. This orthogonality ensures the decomposition
	we find is optimal. In turn, $v$ is characterized by
	the moduli $\sigma=(\lambda,\gamma)$, which are, through the projection lemma
	above, functions of $u$. 
	We call $v$ the adiabatic part of $u$. 
	
	In the remaining sections, we use \lemref{lem2.1} to decompose an entire flow $u_t$ starting near $M$ under \eqref{1.1} into an adiabatic
	flow $v_t$, of which we have explicit information, and a uniformly small
	remainder. Then we check that the concentration set $\gamma_t$ associated
	to $u_t$ indeed evolves according to the binormal curvature flow.

	\section{Effective dynamics}

	\subsection{The connection of \eqref{1.1} to the binormal curvature flow}
	
	The following lemma translates  the r.h.s. of \eqref{1.1} restricted to 
	$M$ to an expression on the tangent bundle $T\Sigma_{\eps^\alpha}$. 
	Afterwards, we show this defines an evolution on $C^k_\text{per}$
	that agrees with the binormal curvature flow in the leading
	order.
	
	In this subsection, the operator $J$ denotes multiplication by
	the standard symplectic matrix \eqref{4.0}.
	
	\begin{lemma}
		\label{lem3.1}
		Let $\alpha>2$.	If $\sigma\in \Sigma_{\epsilon^\alpha}$, then 
		\begin{equation}
			\label{3.2}
			\J_\sigma^{-1}d_\sigma E(f(\sigma))=\del{o(\epsilon^\alpha),J\g_{zz}+o_{\norm{\cdot}_{C^0}}(\epsilon^\alpha)}.
		\end{equation}
	\end{lemma}
	\begin{proof}
		Recall $\sigma=(\lambda,\gamma)$ consists of a $U(1)$-gauge parameter $\lambda$
		and a concentration
		curve $\g$. 
		The $\lambda$-component is therefore not
		physically relevant, and we are interested in the $\gamma$ component.
		For $\sigma\in Y^k$, we write $\sigma=:([\sigma]_\lambda,[\sigma]_\g)$.
		
		The
		partial-Fr\'echet derivative
		of the energy $E(f(\sigma))$ w.r.t.  $\gamma$ is given by 
		\begin{equation}
			\label{3.2.1}
			\di_\gamma E(f(\sigma))=\int_x\inn{\grad_x \psi_\g \cdot \gamma_{zz}+\grad^2_x\psi_\g\gamma_z\cdot \gamma_z}{\grad_x\psi_\g},
		\end{equation}
		where we write
		$\gamma_z=\di_z\gamma$, etc..
		Using this and  the expression for $\J_\sigma$ given in \eqref{A.3}, we find
		that the first term in the r.h.s. of \eqref{3.2.1}
		$$[\J_\sigma^{-1}\int_x\inn{\grad_x \psi_\g \cdot \gamma_{zz}}{\grad_x\psi_\g}]_\gamma=J\gamma_{zz}+o(\epsilon^\alpha).$$ 
		It suffices then to show the second term
		\begin{equation}
			\label{3.3}
			\norm{[\J_\sigma^{-1}\int_x\inn{\grad^2_x\psi_\g\gamma_z\cdot \gamma_z}{\grad_x\psi_\g}]_\gamma}_{C^0}=o(\epsilon^\alpha),
		\end{equation}
		i.e. this term can be absorbed into the remainder. 
		
		First, we use \eqref{2.0.5} and \eqref{4.3} to get 	 \begin{equation}\label{3.3.0}
			\begin{aligned}
				\norm{\int_x\inn{\grad^2_x\psi_\g\gamma_z\cdot \gamma_z}{\grad_x\psi_\g}}_{C^0}&\le C(\Omega)\norm{\grad_x\psi_\g}_{X^0}\norm{\gamma_z}_{C^0}^2\norm{\grad^2_x\psi_\g}_{L^\infty(\Omega)}\\&\le C(\Omega)\epsilon^{2\alpha}\abs{\log\epsilon}^{1/2} \norm{\grad^2_x\psi_\g}_{L^\infty(\Omega)}.
			\end{aligned}
		\end{equation}
		
		Next, note the function $\psi_\g$ satisfies the linear second order equation
		$$\Lap\psi+\grad_x^2\psi\gamma_z\cdot\gamma_z+\grad_x\psi\cdot\gamma_z+\frac{1}{\epsilon^2}(1-\abs{\psi_\g}^2)\psi=0\quad \text{on }\Omega.$$
		This equation is elliptic for $\norm{\gamma_z}_{C^0}\ll1$, so the Schauder
		estimate implies for all $\epsilon<\epsilon_0\ll1$ and $0<\mu<1$,
		\begin{equation}
			\label{3.3.00}
			\begin{aligned}
				\norm{\psi_\g}_{C^{2,\mu}(\Omega)}&\le C(\Omega, \epsilon_0)\del{\norm{\psi_\g}_{L^\infty(\Omega)}+\norm{\frac{1}{\epsilon^2}(1-\abs{\psi_\g}^2)}_{C^{0,\mu}(\Omega)}}\\&=C(\Omega, \epsilon_0)\epsilon^{-2(1+\mu)}.
			\end{aligned}
		\end{equation}
		Here we have used the estimates on the uniform norm $\norm{\psi_\g}_{L^\infty(\Omega)}\le1$ and H\"older seminorm
		$\sbr{\psi_\g}_{\mu,\Omega}\le C(\Omega)\epsilon^{-\mu}.$
		These follow from \eqref{2.0.2} and \eqref{4.2}.
		
		Plugging \eqref{3.3.00} to \eqref{3.3.0}, we see that 
		\begin{equation}\label{3.3.000}
			\norm{\int_x\inn{\grad^2_x\psi_\g\gamma_z\cdot \gamma_z}{\grad_x\psi_\g}}_{C^0}\le C(\Omega, \epsilon_0)\abs{\log\epsilon}^{1/2} \epsilon^{2(\alpha-1-\mu)}.
		\end{equation}
		In Appendix we show
		$\norm{\J_\sigma^{-1}}_{Y^k\to Y^k}\le C(\omega)\abs{\log\epsilon}^{-2}$.
		This, the assumption $\alpha>2$, and the above estimate  together show that if we  choose $\mu<\alpha/2-1$
		then \eqref{3.3} holds. The proof is complete.
	\end{proof}
	
	We now discuss the geometric meaning of  the r.h.s. of \eqref{3.2}.
	A family of curves $\vec{\gamma}_t\in C^2(I,\Omega)$ satisfies the \emph{binormal curvature flow}
	if parametrized by arclength, the curves satisfy
	\begin{equation}
		\label{BNF}
		\di_t\vec{\gamma}=\vec{\gamma}_s\times\vec{\gamma}_{ss}\quad \del{\abs{\vec{\gamma}_s}\equiv 1}.
	\end{equation}
	Here we write derivatives w.r.t. to the arclength parameter as $\vec{\gamma}_s\equiv \di_s\vec{\gamma}$, etc..
	
	Now reparametrize $\vec{\gamma}_t$ as $\vec{\gamma}_t(z)=(\gamma_t(z),z)$ with $\gamma\in C^2(I,\omega)$. Then
	\eqref{BNF} becomes
	\begin{equation}
		\label{3.3.1}
		\di_t\vec{\gamma}=\vec{\gamma}_zz_s\times(\vec{\gamma}_{zz}z_s^2+\vec{\gamma}_z z_{ss}).
	\end{equation}
	The arclength parameter $s$ is given in terms of the new parameter $z$ by
	$$s(z)=\int_0^z\abs{\vec{\gamma_z}}^2=z+\int_0^z\abs{\gamma_z}^2.$$
	Differentiating this expression, for $\norm{\gamma}_{C^2}=O(\delta)$, we get 
	$$z_s=\frac{1}{\abs{\gamma_z}^2+1}=1+O(\delta^2),\quad z_{ss}=-\frac{\gamma_z\cdot \gamma_{zz}}{\del{1+\abs{\gamma_z}^2}^3}=O(\delta^2).$$
	Geometrically, $z$ and $s$ are close because the curve parametrized by
	$\vec{\gamma}$ with $\norm{\gamma}_{C^2}\ll1$ is approximately 
	a vertical straight line, in which case $z=s$. 
	
	Thus \eqref{3.3.1} in the leading order reads
	\begin{equation}
		\label{3.4}
		\di_t\vec\gamma=\vec\gamma_z\times \vec\gamma_{zz}+O(\delta^2)=(J\gamma_{zz},\gamma_z^1\gamma_{zz}^2-\gamma_z^2\gamma_{zz}^2)+O(\delta^2).
	\end{equation}
	In the r.h.s. of \eqref{3.4}, $\gamma_z^1\gamma_{zz}^2-\gamma_z^2\sigma_{zz}^2$ is also $O(\delta^2)$.
	In conclusion, combing with \lemref{lem3.1}, we can say
	$$\text{$\di_t\sigma=(\di_t\lambda,\di_t\gamma)=\J_\sigma^{-1}d_\sigma E(f(\sigma))+o(\delta)\implies\vec{\gamma}_t$ solves \eqref{BNF} up to $o(\delta)$}.$$
	
	In the next subsection, we show that this equation for $\sigma$ is indeed the effective 
	dynamics of vortex filaments.
	
	\subsection{Proof of the first main theorem}
	Suppose $u_t$ is a solution to \eqref{1.1} such that $\dist(u_t, M)<\delta\ll1$
	for $t\le T$.
	So far, using \lemref{lem2.2}, this allows us to define an adiabatic flow $v_t=f(\sigma_t)\in M$ consisting of the approximate
	filaments associated to $u_t$.
	At time $t\le T$, the filament $v_t$ is characterized by the moduli $\s_t$, and the curve  $\gamma_t=[\sigma_t]_\g$ defines the concentration set
	of $u_t$ (see Remark \ref{concentrationset}).

	In what follows, we show that the velocity 
	$\di_t\sigma$ governing the motion of the adiabatic flow is given by $\J_\sigma^{-1}d_\sigma E(f(\sigma))$ uniformly up to the leading order. 
	We then find an a priori estimate for the remainder, so that 
	as long as $u_0$ is close to $M$, the full flow $u_t$ remains uniformly close to $M$ up to a large time. In this sense one can view $M$ as an invariant manifold for \eqref{1.1}.

	
	\begin{theorem}[effective dynamics]
		\label{thm3.1}
		For any $\epsilon>0$, there are $\delta_1,\,\delta_2\ll\epsilon$ such that the following holds:
		Let $\Sigma_{\delta_1}\to M$ be the manifold of approximate vortex filaments as in Section 2.
		Let $u_0$ be an initial configuration such that $\dist_{X^2}(u_0, M)<\delta_2$.

		Then there is some $T>0$ independent of $\epsilon$ and $\delta_1,\,\delta_2$, 
		such that for all $\epsilon t\le T$, 	there exists moduli
		$\s_t$ associated to $u_t$ as in \lemref{lem2.2}, and 
		\begin{equation}
			\label{3.4.1}
			\norm{u_t-f(\s_t)}_{X^2}=o(\sqrt{\delta_1}).
		\end{equation}
		
		Moreover,  for all $\epsilon t\le T$,  the moduli $\s_t$
		evolves according 
		to 
		\begin{equation}
			\label{3.4.2}
			\di_t\sigma=\J_\sigma^{-1}d_\sigma E(f(\sigma))+o_{\norm{\cdot}_{Y^0}}(\delta_1).
		\end{equation}
	\end{theorem}
	\begin{proof}			
		1. To begin with, choose $\delta_2\ll \sqrt{\delta_1}$.
		We show $w_0:=u_0-f(\sigma_0)=u_0-v_0$ satisfies $\norm{w_0}_{X^2}\ll \sqrt{\delta_1}$.
		
		Suppose the $X^2$-closest approximate vortex filament to $u_0$ is $v_*=f(\s_*)\in M$.
		Then $w_0=(u_0-v_*)+(v_*-v_0)$. The first difference on the r.h.s.
		has size $\delta_2$. The second difference can be bounded by
		$\norm{v_*-v_0}_{X^2}=\norm{ f(\s_*)-f(\s_0)}_{X^2}\le C \norm{\s_*-\s_0}_{Y^2}\le C \delta_2.$
		It follows that $\norm{w_0}_{X^2}\le C \delta_2$, and by the choice of $\delta_2\ll\sqrt{\delta_1}$,
		that $\norm{w_0}_{X^2}\ll \sqrt{\delta_1}$. Therefore,
		by the continuity of the evolution \eqref{1.1},
		\eqref{3.4.1} holds at least locally.

		2. So long as the decomposition $u_t=v_t+w_t$ in \lemref{lem2.2} is valid, we can expand \eqref{1.2} as
		\begin{equation}
			\label{3.5}
			\di_tv+\di_tw=J(E'(v)+L_\sigma w+N_\sigma(w)),
		\end{equation}
		where $J:\psi\mapsto -i\psi$, 
		\begin{equation}
			\label{3.6}
			L_\sigma\phi:=-\Lap \phi+\frac{1}{\epsilon^2}(\abs{\psi_\g}^2-1)\phi+\frac{2e^{i\lambda}\cos \lambda }{\epsilon^2}\psi_\g\inn{\psi_\g}{\phi}
		\end{equation}
		is the linearized operator at $f(\s)\equiv v$, and $$N_\sigma(\phi):=E'(\psi_\g+\phi)-E'(\psi_\g)-L_\sigma\phi$$ is the 
		nonlinearity. 
		
		For the moduli $\sigma=\sigma_t$ associated to
		$u_t$ given in \lemref{lem2.2}, let $Q_\sigma:X^0\to X^0$ be the fibrewise projection onto $T_{f(\sigma)}M$, given in \eqref{2.4}. 
		Applying $Q_\sigma$ to both sides of \eqref{3.5}, we have
		\begin{equation}
			\label{3.7}
			\di_tv-Q_\sigma JE'(v)=Q_\sigma(JL_\sigma w -\di_tw+JN_\sigma(w)).
		\end{equation}
		
		Consider the identity
		$$\J_\sigma^{-1}g_\sigma^*J^{-1}(\di_tv-Q_\sigma JE'(v))=\di_t\sigma-\J_\sigma^{-1}d_\sigma E(f(\sigma)).$$
		To verify this, one uses two facts that follow readily from the chain rule:
		$$\di_tv=g_\sigma\di_t\sigma,\quad g_\sigma^*E'(f(\sigma))=d_\sigma E(f(\sigma)).$$
		Thus by the uniform estimates on $g_\sigma^*$ and $\J_\sigma^{-1}$ in \eqref{A.2}, \eqref{A.6}, we have 
		\begin{equation}
			\label{3.8}
			\norm{\di_t\sigma-\J_\sigma^{-1}d_\sigma E(f(\sigma))}_{Y^0}\le C\abs{\log\epsilon}^{-3/2}\norm{\di_tv-Q_\sigma JE'(v)}_{X^2},
		\end{equation}
		for some $C$ independent of $\sigma$. 
		This shows that the claim \eqref{3.4.2} would follow if we have uniform control over
		\eqref{3.7}.
		
		3. We now derive the a priori estimate for the remainder $w=w_t$,
		the fluctuation of $u_t$ around the adiabatic part $v_t$.
		
		Consider the r.h.s. of \eqref{3.7}. 
		These three terms can be bounded respectively as follows:
		\begin{align}
			&\norm{Q_\sigma JL_\sigma w}_{X^0}\le C\abs{\log\epsilon}^{-1}\delta_1^{1/2}\norm{w}_{X^2}\label{3.9},\\
			&\norm{Q_\sigma\di_tw}_{X^0}\le C\abs{\log\epsilon}^{-1}\norm{\di_t\sigma}_{Y^0}\norm{w}_{X^1},\label{3.10}\\
			&\norm{Q_\sigma JN_\sigma(w)}_{X^0}\le C\abs{\log\epsilon}^{-1}\epsilon^{-2}\norm{w}_{X^2}^2.\label{3.11}
		\end{align}
		Here $C$ is independent of $\sigma$. 
		In all these three inequalities we use the uniform bound $\norm{Q_\sigma}_{X^0\to X^0}\le C(\Omega)\abs{\log\epsilon}^{-1}$,
		which follows its definition \eqref{2.4} and the estimates 
		for each of its factors.
		
		First, we show \eqref{3.9}
		using the identity 
		\begin{equation}\label{3.200}	
			\abs{\inn{Q_\sigma JL_\sigma w}{w'}}= \abs{\inn{w}{L_\sigma Q_\sigma Jw'}}
			\quad (w,w'\in X^2).
		\end{equation}
		To get \eqref{3.200}, one uses the relation
		$Q_\s J=JQ_\s^*$, which follows from the definition of $Q_\s$ (see also 
		the first step in the proof of \lemref{lem2.2}).
		Plugging $w'=Q_\sigma J_\sigma Lw$ into \eqref{3.200},
		and using estimate \eqref{4.1.2}, we find 
		$$\norm{Q_\sigma JL_\sigma w}_{X^0}\le  C(\delta_1^{1/2}\norm{w}_{X^2}).$$
		This gives \eqref{3.9}.
		
		Next, differentiating $Q_\sigma w=0$ w.r.t.  $t$, we have
		\begin{equation}\label{3.10.1}
			0=\di_t(Q_\sigma w)=(\di_t Q_\sigma)w+Q_\sigma\di_tw=(d_\sigma Q_\sigma\di_t\sigma)w+Q_\sigma\di_tw.
		\end{equation}
		Here $d_\sigma Q_\sigma$ is an operator from $Y^0$ to the space of linear
		operators $L(X^1, X^0)$.  
		Since $Q_\sigma$ is the projection onto $T_\sigma M$,
		and $\sigma$ is slowly varying, we get the uniform estimate
		$\norm{d_\sigma Q_\sigma}_{Y^0\to L(X^1, X^0)}\le C.$ 
		Plugging this to \eqref{3.10.1} gives \eqref{3.10}.
		
		Lastly, we have the following explicit expression for the nonlinearity: 
		$$N_\sigma(w)=\frac{1}{\epsilon^2}(2w\inn{ v}{w}+\abs{w}^2(v+w)).$$
		Using \eqref{2.0.4} we can bound this as \eqref{3.11}.

		Plugging \eqref{3.9}-\eqref{3.11} to \eqref{3.7},
		and then using \eqref{3.8}, one can see that 
		\eqref{3.4.2} follows if one can bound $\norm{w}_{X^2}$ uniformly 
		in time by a sufficiently small number compared to $\epsilon$.
		It suffices then to show \eqref{3.4.1} with $\delta_1\ll\epsilon$.
		
		4. Consider the expansion
		\begin{equation}
			\label{3.12}
			E(v+w)=E(v)+\inn{E'(v)}{w}+\frac{1}{2}\inn{L_\sigma w}{w}+R_\sigma(w),
		\end{equation}
		where $R_\sigma(w)$  is the super-quadratic remainder defined by this expression.

		We want to apply  the coercivity of $L_\sigma$ shown in
		\lemref{lem4.2}  to $w$, so that we can rearrange \eqref{3.12} to get 
		\begin{equation}
			\label{3.12.1}
			\norm{w}_{X^2}^2\le C_\epsilon(E(v+w)-E(v)-\inn{E'(v)}{w}-R_\sigma(w)).
		\end{equation}
		The constant $C_\epsilon$ will eventually be eliminated by the choice of 
		$\delta$.
		This requires the uniform estimates \begin{equation}
			\label{claim2}\norm{w}_{X^2}\le C_\epsilon \norm{w}_{X^0}. 
		\end{equation}
		For this, one uses the fact that $w$ satisfies the 
		elliptic equation
		$$\Lap w+\frac{1}{\epsilon^2}(1-\abs{v}^2)w=-i\di_t u+\grad_xv\cdot
		\g_{zz}-\grad_x^2v_z\cdot\g_z+(\abs{u}^2-\abs{v}^2)u.$$
		The r.h.s. of this equation is $O(\delta_1)$. Thus standard 
		elliptic regularity gives 
		$\norm{w}_{X^2}\le C(\Omega,\epsilon)\norm{w}_{X^0}$, with
		$C(\Omega,\epsilon)=O(\epsilon^{-2})$. 
		
		Since $u=v+w$ solves \eqref{1.1}, the energy $E(u)$ is conserved for all $t$.
		Consequently, 
		$$E(v+w)=E(v_0+w_0)=E(v_0)+\inn{E'(v_0)}{w_0}+\frac{1}{2}\inn{L_\sigma w_0}{w_0}+R_\sigma(w_0).$$
		Plugging this into \eqref{3.12.1}, and using \eqref{claim2}, one gets
		\begin{equation}
			\label{3.12.2}
			\begin{aligned}
				\norm{w}_{X^2}^2&\le C_\epsilon\alpha^{-1}(E(v_0)-E(v)\\&+\inn{E'(v_0)}{w_0}-\inn{E'(v)}{w}+\frac{1}{2}\inn{L_\sigma w_0}{w_0}-R_\sigma(w_0)+R_\sigma(w)).
			\end{aligned}
		\end{equation}
		Here $\alpha=\alpha(\epsilon)=O(\abs{\log\epsilon}^{-1})$ is as in \lemref{lem4.2}.
		
		Similar to the nonlinear estimate on $N_\sigma(w)$, it follows 
		from the regularity of $E$ on the energy space 
		that $R_\sigma(w)\le C\epsilon^{-2}\norm{w}_{X^2}^3$ where $C$ is independent of $\sigma$. Combining this with \lemref{lem4.1},  we can bound 
		\begin{equation}
			\label{3.101}
			\begin{aligned}
				&\inn{E'(v_0)}{w_0}-\inn{E'(v)}{w}+\frac{1}{2}\inn{L_\sigma w_0}{w_0}-R_\sigma(w_0)+R_\sigma(w)\\
				\le& C\Bigl(\norm{E'(v_0)}_{X^0}\norm{w_0}_{X^0}+\epsilon^{-1}\norm{w_0}_{X^1}^2+\epsilon^{-2}\norm{w_0}_{X^2}^3\\&+
				\norm{E'(v)}_{X^0}\norm{w}_{X^0}+\epsilon^{-2}\norm{w}_{X^2}^3\Bigr).
			\end{aligned}
		\end{equation}

		Recall in Step 1 we have shown $\norm{w_0}_{X^2}\le C  \dist_{X^2}(u_0,M)=C \delta_2$.
		So \eqref{3.101} becomes 
		\begin{equation}
			\label{3.102}
			\begin{aligned}
				&\inn{E'(v_0)}{w_0}-\inn{E'(v)}{w}+\frac{1}{2}\inn{L_\sigma w_0}{w_0}-R_\sigma(w_0)+R_\sigma(w)\\
				\le& C\del{\norm{E'(v_0)}_{X^0}\delta_2+\epsilon^{-1}\delta_2^2+\epsilon^{-2}\delta_2^3+\norm{E'(v)}_{X^0}\norm{w}_{X^0}+\epsilon^{-2}\norm{w}_{X^2}^3}.
			\end{aligned}
		\end{equation}
		This gives control over the last five terms in the r.h.s. of \eqref{3.12.2}
		
		5. We now exploit 
		energy conservation to control the first two terms in the r.h.s. of\eqref{3.12.2}, as
		this difference is the energy fluctuation of the approximate filaments.
		Differentiate the energy $E(t)=E(f(\sigma_t))$ and using \eqref{3.7}, we have
		\begin{equation}
			\label{3.13}
			\begin{aligned}
				\od{E}{t}&=\inn{E'(v)}{\di_tv}\\
				&=\inn{E'(v)}{Q_\sigma JE'(v)}+\inn{E'(v)}{Q_\sigma (JL_\sigma w-\di_tw)}+\inn{E'(v)}{Q_\sigma JN_\sigma(w))}.
			\end{aligned}
		\end{equation}
		We now bound the three inner products respectively. 
		
		Using the relation $Q_\s J=JQ_\s^*$, the property of projection $Q_\s^2=Q_\s$,
		and the fact that $J$ is symplectic, we find
		$$\begin{aligned}
			\inn{E'(v)}{Q_\sigma JE'(v)}&=\inn{E'(v)}{Q_\sigma^2 JE'(v)}\\
			&=\inn{Q_\sigma ^*E'(v)}{Q_\sigma JE'(v)}\\
			&=\inn{(J^{-1}J)Q_\sigma^* E'(v)}{JQ_\sigma E'(v)}\\
			&=-\inn{JQ_\sigma J E'(v)}{Q_\sigma J E'(v)}=0.
		\end{aligned}$$
		Thus the first term in \eqref{3.13} vanishes. 
		
		Using \eqref{3.9}-\eqref{3.10}, the second inner product can be bounded as
		$$\abs{\inn{E'(v)}{Q_\sigma(JL_\sigma w-\di_tw)}}\le\norm{E'(v)}_{X^0} C(\delta_1^{1/2}+\norm{\di_t\sigma}_{Y^0})\norm{w}_{X^2}.$$
		By \eqref{3.2},\eqref{3.7}-\eqref{3.11}, so long as $\norm{w}_{X^1}<1/2$ 
		we have 
		$$\norm{\di_t\sigma}_{Y^0}\le C(\delta_1+\delta_1^{1/2}\norm{w}_{X^2}+\epsilon^{-2}\norm{w}_{X^2}^2).$$ 
		Plugging this back to the previous estimate, we have
		\begin{equation}
			\label{3.14}
			\begin{aligned}
				&\abs{\inn{E'(v)}{Q_\sigma(JL_\sigma w-\di_tw)}}\\\le&\norm{E'(v)}_{X^0}(C_1\delta_1^{1/2}+C_2(\delta_1+\delta_1^{1/2}\norm{w}_{X^2}+\epsilon^{-2}\norm{w}_{X^2}^2))\norm{w}_{X^2}.
			\end{aligned}
		\end{equation}
		
		Lastly, by the nonlinear estimate \eqref{3.11},
		the third inner product in \eqref{3.13} can be bounded as
		\begin{equation}
			\label{3.15}
			\abs{\inn{E'(v)}{Q_\sigma JN_\sigma(w))}}\le C\epsilon^{-2}\norm{E'(v)}_{X^0}\norm{w}_{X^2}^2.
		\end{equation}
		
		6. Combining \eqref{3.13}-\eqref{3.15} and integrating from $0$ to $t$, we have
		\begin{equation}
			\label{3.16}
			\begin{aligned}
				&\abs{E(v_t)-E(v(0))}\\\le& t\norm{E'(v)}_{X^0}(C_1\delta_1^{1/2}+C_2(\delta_1+\delta_1^{1/2}M(t)+\epsilon^{-2}M(t)+\epsilon^{-2}M(t)^2))M(t),
			\end{aligned}
		\end{equation}
		where $M(t):=\sup_{t'\le t}\norm{w(t')}_{X^2}$. 
		Plugging \eqref{3.102} and \eqref{3.16} into
		\eqref{3.12.2}, we have 
		
		\begin{equation}
			\label{3.17}
			\begin{aligned}
				M(t)&\le \abs{\log\epsilon}C_\epsilon \norm{E'(v)}_{X^0} t\del{C_1\delta_1^{1/2}+C_2(\delta_1+\delta_1^{1/2}M(t)+\epsilon^{-2}M(t)+\epsilon^{-2}M(t)^2)}\\
				&+C_3\del{ \norm{E'(v)}_{X^0}(1+\epsilon^{-2}M(t)^2)+\norm{E'(v_0)}_{X^0}\delta_2+\epsilon^{-1}\delta_2^2+\epsilon^{-2}\delta_2^3}.
			\end{aligned}
		\end{equation}
		If we now choose $\delta_2=\min(\abs{\log\epsilon}^{-1/2},\epsilon\delta_1^{1/2})$ and the manifold of approximate filaments $\Sigma_{\delta_1}\to M$ to be sufficiently small, say with
		\begin{equation}\label{3.18.1}\delta_1=O(\abs{\log\epsilon}^{-2}C_\epsilon^{-1}\epsilon^{1+\mu})
		\end{equation} for some $\mu>0$, 
		then \eqref{3.17} and \lemref{lem2.1} implies that
		for all $t\le T=C \epsilon^{-1}$, we have $M(t)\le C\epsilon^\mu\sqrt{\delta}$, as claimed in \eqref{3.4.1}. (Actually, if we further shrink $\delta_1$ (i.e. the class
		of initial configurations), we can ensure 
		\eqref{3.4.1} on a  longer time interval. )
		
		Plugging \eqref{3.4.1} into \eqref{3.7}-\eqref{3.11}, we have 
		$$\norm{\di_t\sigma-\J_\sigma^{-1}d_\sigma E(f(\sigma))}_{Y^0}\le C\abs{\log\epsilon}^{-5/2}\epsilon^{-2}\delta_1^2.$$ 
		Shrinking $\delta_1$, say with $\mu>1$ in \eqref{3.18.1}, we can ensure  this expression is still $o(\delta_1)$. The proof is complete.
	\end{proof}
	
	\subsection{\thmref{thm3.1} relates to Jerrard's conjecture}
	
	The following theorem refers to \cite[Conjecture 7.1]{MR3729052}.
	For small but fixed $\epsilon$, \thmref{thm3.2} below, in particular \eqref{3.18}-\eqref{3.19}, provides an affirmative answer for a certain class of initial conditions
	with uniformly small curvature.
	
	\begin{theorem}\label{thm3.2}
		For any $\beta>0$, there exist $\delta,\epsilon_0>0$ such that the following holds: Let $\epsilon<\epsilon_0$ in \eqref{1.1}. 
		Let $\sigma_0\in\Sigma_\delta$ be given, and $\gamma_0=[\sigma_0]_\gamma$.
		Let $\vec\g_t$ be
		the flow generated by $(\gamma_0(z),z)$ under \eqref{BNF}.
		
		Then there exist a solution $u_t$ to \eqref{1.1},
		and  some $T>0$ independent of $\epsilon$ and $\delta$, 
		such that for all $\epsilon t\le T$, $X\in C^1_c(\Rr^3,\Rr^3)$ and $\phi\in C^1_c(\Rr^3,\Rr)$ with $\norm{X}_{C^1},\,\norm{\phi}_{C^1}=O(\delta^{-1/4})$, we have
		\begin{align}
			&\abs{\int_\Omega X\times Ju-\pi\int_{\vec\g_t}X}\le \beta,\label{3.18}\\
			&\abs{\int_\Omega \frac{e(u)}{\abs{\log\epsilon}}\phi-\pi\int_{\vec\g_t}\phi\,d \H^1}\le \beta  \label{3.19},
		\end{align}
		where for $\psi=\psi^1+i\psi^2$,
		$$
		e(\psi)=	\frac{1}{2}\abs{\grad\psi}^2+\frac{1}{4\epsilon^2}(\abs{\psi}^2-1)^2,\quad  J\psi=\grad\psi^1\times \grad\psi^2.$$
	\end{theorem}
	
	
	\begin{proof}
		Let $u_t$ be the flow
		generated by
		$u_0:=f(\sigma_0)$ under \eqref{1.1}. Then for sufficiently small $\delta>0$, \thmref{thm3.1} applies with $\delta_1=\delta,\,\delta_2=0$.
		Let $\tilde{\sigma}_t$ be the flow generated by $\sigma_0$ under the effective dynamics \eqref{3.4.2}.
		It follows that 
		$u_t=v_t+O_{\norm{\cdot}_{X^2}}(\sqrt{\delta})$, where $v_t=f(\tilde{\sigma_t})$. 
		
		Given the explicit construction, one can check using classical  concentration properties of the planar vortex $\psi^{(1)}$ (for instance \cite{MR1491752})
		that, for sufficiently small
		$\epsilon_0=\epsilon_0(\beta)>0$ and all $0<\epsilon\le\epsilon_0$,
		the flow $v_t$ satisfies 
		\eqref{3.18}-\eqref{3.19}, with $\tilde{\sigma}_t$ in place of $\sigma_t$.
		For example, we can compute using \eqref{4.3} and the assumption $\norm{\phi}_{C^0}=O(\delta^{-1/4})$ that
		$$\begin{aligned}
			\int_\Omega \frac{e(v_t)}{\abs{\log\epsilon}}\phi&=\int_{\vec{\tilde{\gamma}}_t}\del{\int_\omega
				\frac{e(\psi_{\tilde{\sigma}_t})}{\abs{\log\epsilon}}\phi\,dx }\,d\H^1\\
			&\le\int_{\vec{\tilde{\gamma}}_t}\del{\int_\omega
				\frac{e(\psi^{(1)})}{\abs{\log\epsilon}}\phi\,dx }\,d\H^1+C\frac{\epsilon^2\delta}{\abs{\log\epsilon}}\norm{\phi}_{C^0}\\
			&\le \pi\int_{\vec{\tilde{\gamma}}_t}\phi\,d\H^1+C\del{\beta+ \frac{\epsilon^2\delta^{3/4}}{\abs{\log\epsilon}}}.
		\end{aligned}$$
		Here $\tilde{\gamma}=[\tilde{\sigma}]_\gamma$, and $\vec{\tilde{\gamma}}(z)=(\tilde{\gamma(z)},z)$.
		
		Using \eqref{3.3.000} with $\alpha=3$ and $\mu=1/4$, we can get the uniform estimate  $\tilde{\sigma}_t=\sigma_t+O(\delta^{3/2})$.
		If $\norm{X}_{C^1}=O(\delta^{-1/4})$, then by the mean value theorem,
		$$\abs{\int_{\vec{\gamma}} X-\int_{\vec{\tilde{\gamma}} }X}\le \norm{\g-\tilde{\g}}_{L^{\infty}(I)}\norm{X}_{C^1}=O(\delta^{5/4}).$$
		Similarly one can show 
		$$\abs{\int_{\vec{\gamma}} \phi\,d\H^1-\int_{\vec{\tilde{\gamma} }}\phi\,d\H^1}=O(\delta^{5/4}).$$
		
		It follows that for all sufficiently small $\eps_0$ and all $\epsilon<\epsilon_0$, 
		\begin{align}
			&\abs{\int_\Omega X\times Jv-\pi\int_{\g_t}X}\le C\del{\beta+ \frac{\epsilon^2\delta^{3/4}}{\abs{\log\epsilon}}+\delta^{5/4}},\label{3.20}\\
			&\abs{\int_\Omega \frac{e(v)}{\abs{\log\epsilon}}\phi-\pi\int_{\g_t}\phi\,d \H^1}\le C\del{\beta+ \frac{\epsilon^2\delta^{3/4}}{\abs{\log\epsilon}}+\delta^{5/4}}\label{3.21},
		\end{align}
		
		On the other hand, using \eqref{3.4.1} and the continuity of $e(\psi),\,J\psi$ on $X^1$,  we have
		$$\abs{\int_\Omega X\times Ju-\int_\Omega X\times Jv}\le
		C\norm{X}_{C^0(\Omega)}\norm{u-v}_{X^2}\le C\delta^{1/4},$$
		and similarly 
		$$\abs{\int_\Omega\frac{e(u)}{\abs{\log\epsilon}}\phi\, d \H^1-\int_\Omega\frac{e(v)}{\abs{\log\epsilon}}\phi\,d \H^1}\le\ C\abs{\log\epsilon}^{-1}\delta^{1/4}.$$
		Plugging these into \eqref{3.20}-\eqref{3.21}, we get that 
		$u_t$ satisfies \eqref{3.18}-\eqref{3.19} for sufficiently small $\delta=\delta(\beta, \epsilon_0)>0$.
	\end{proof}

	\section{Properties of the linearized operators}
	In this section we consider various estimates for the linearized operator
	$L_\sigma=E''(f(\sigma))$ defined in \eqref{3.6}.  Recall that so far we have always suppressed the dependence of
	the various functions on the material parameter $\epsilon\ll1$.
	Through out this section
	we assume $\norm{\sigma}_{Y^2}<\delta \ll\eps$.
	Without specification, various inner products are as in \eqref{2.00X}.

	\begin{lemma}[uniform bound of $L_\sigma$]
		\label{lem4.1}
		There exists $0<C<\infty$ independent of $\sigma$ and $\epsilon$ such that
		\begin{align}
			\label{4.1.1}
			&\inn{L_\s\phi}{\phi}\le C\epsilon^{-1}\norm{\phi}_{X^0}^2\quad (\phi\in X^1),\\
			\label{4.1.2}
			&\norm{L_\sigma Q_\sigma\phi}_{X^0}\le C\delta^{1/2}\norm{\phi}_{X^2} \quad 
			(\phi \in X^2).		\end{align}
		Here, and 
		$Q_\s$ is the projection onto the tangent space $T_{f(\s)}M$ defined in \eqref{2.4}.
	\end{lemma}
	\begin{proof}
		We write the Schr\"odinger operator $L_\s$ 
		defined in \eqref{3.6} as $L_\s=-\Lap+V$.
		By Poincar\'e's inequality, the kinetic part of \eqref{4.1.1} 
		is bounded by
		$\inn{-\Lap\phi}{\phi}\le C \norm{\phi}_{X^0}$.
		To bound the potential part of \eqref{4.1.1},
		consider 
		$$
		\begin{aligned}
			\norm{V(\phi)}_{X^0}&=\epsilon^{-2}\norm{(\abs{\psi_\g}^2-1)\phi+2e^{i\lambda}
				\cos \lambda\psi_\g\Re(\cl{\psi_\g}\phi)}_{X^0}\\ 
			&\le
			\epsilon^{-2}\del{\norm{\abs{\psi_\g}^2-1}_{X^0}+2\norm{\psi_\g^2}_{X^0}}\norm{\phi}_{X^0}\\
			&\le\epsilon^{-2}\del{\norm{\abs{\psi_\g}^2-1}_{X^0}+\diam(\omega)^2\norm{\psi_\g}_{L^\infty(\Omega)}}\norm{\phi}_{X^0}\\
			&\le (C(\omega)\epsilon^{-1}+\diam(\omega)^2)\norm{\phi}_{X^0}.
		\end{aligned} $$
		Recall $\omega\subset \Rr^2$ is the cross section of the domain $\Omega$.
		In the last step we use \eqref{2.0.4}, \eqref{2.0.3}, and \eqref{4.2}.
		This implies $\inn{V(\phi)}{\phi}\le C\eps^{-1} \norm{\phi}_{X^0}^2$,
		and therefore \eqref{4.1.1} follows.

		Next, we show \eqref{4.1.2}.
		Since $Q_\s$ is the projection onto the tangent space $T_{f(\s)}M$,
		using the formula  \eqref{A.1} for the trivialization, every 
		$\phi\in\ran Q_\sigma$ can be written as
		\begin{equation}\label{phi}
			\phi=e^{i\lambda}\del{i\mu\psi_\g-\grad_x\psi_\g\cdot \xi}
		\end{equation} for some $\xi \in C^0_\text{per}$
		and $\mu\in\Rr$.
		Using this representation, the assumption $\lambda<\delta$ (which implies
		$e^{i\lambda}=1+iO(\delta)$), together with the formula \eqref{3.6} 
		for $L_\s$ (in which only the last term is not real-linear),
		we can write
		\begin{equation}
			\label{4.1.3}
			L_\sigma Q_\sigma\phi=L_\s(i\mu\psi_\g -\grad_x\psi_\g\cdot \xi)+O(\delta\eps^{-2}\phi).
		\end{equation}
		
		For fixed $\gamma \in C^2_\text{per}$,
		let $L_{z,x}:H^2(\omega)\to L^2(\omega)$ be the planar linearized operator at  $\psi_{\gamma,z}:=f(\sigma)(\cdot,z)$,
		given explicitly as
		\begin{equation}
			\label{4.1.4}
			L_{z,x}\phi:=-\Lap_x \phi+\frac{1}{\epsilon^2}(\abs{\psi_{\gamma,z}}^2-1)\phi+\frac{2}{\epsilon^2}\psi_{\gamma,z}\inn{\psi_{\gamma,z}}{\phi}\quad (\phi:\omega\subset \Rr^2\to \Cc).
		\end{equation}
		
		Consider the inner product $\inn{L_\s Q_\s \phi}{\phi'}$
		for $\phi,\,\phi'\in X^2$.
		Using \eqref{4.1.3}-\eqref{4.1.4}, we can split this into three parts,
		\begin{equation}
			\label{4.1.5}
			\begin{aligned}
				\inn{L_\s Q_\s \phi}{\phi'}&=\int_I\del{\int_\omega \inn{L_{z,x}(i\mu\psi_\g(x,z) -\grad_x\psi_\g(x,z)\cdot \xi(z))}{\phi'(x,z)}^2\,dx}\,dz\\
				&+
				\inn{\di_{zz} (i\mu\psi_\g -\grad_x\psi_\g\cdot \xi)}{\phi'}+O(\delta\eps^{-2})\norm{\phi}_{X^2}\norm{\phi'}_{X^2}.
			\end{aligned}
		\end{equation}
		
		The operator $L_{z,x}$ is obtained by translating the planar
		linearized operator $L^{(1)}_x$ at the simple vortex $\psi^{(1)}$ by $x\mapsto x-\gamma(z)$.
		Consequently, for each fixed  $z$, the three vectors
		$\di_{x_j}\psi_\g(x,z),\,j=1,2$, and $i\psi_\g(x,z)$ are in the kernel of $L_{z,x}$.
		(See the discussion in Sec. \ref{sec:2.1} about symmetry zero modes).
		Because of this fact, the first term in the r.h.s. of \eqref{4.1.5} vanishes.
		
		To estimate the second term,
		we compute
		\begin{equation}
			\label{4.1.6}
			\di_{zz} (i\mu\psi_\g -\grad_x\psi_\g\cdot \xi)=-i\mu\grad_x\psi_\g\cdot \g_{zz}+\grad_x^2\psi_\g\g_{zz}\cdot \xi+O(\norm{\g}_{C^2}^2\sum_{i,j=1,2}\di^2_{x_ix_j}\phi).
		\end{equation}
		By the assumption $\norm{\g}_{C^2}=O(\delta)$, \eqref{4.1.6} implies that
		the second term in the r.h.s. of \eqref{4.1.5} is $O(\delta)\norm{\phi}_{X^2}\norm{\phi'}_{X^2}.$
		Lastly, setting $\phi'=L_\s Q_\s \phi$ in \eqref{4.1.5}, we obtain \eqref{4.1.2}  so long as $\delta= O(\eps ^4)$.
	\end{proof}
	\begin{remark}
		Estimate \eqref{4.1.2} shows that elements in $\ran Q_\sigma=T_{f(\sigma)}M$
		are \textit{approximate zero modes} of $L_\sigma$. 
		If one further shrinks $\delta=O(\eps^\alpha)$ with $\alpha>2$, then
		the power in $\delta$ in the r.h.s. of \eqref{4.1.2} can be improved to $(\alpha-2)/\alpha$. 
	\end{remark}
	
	
	\begin{lemma}[coercivity]
		\label{lem4.2}
		There exists $\alpha=O(\abs{\log\epsilon}^{-1})>0$ independent of $\s$ 
		such that
		\begin{equation}
			\label{4.6}
			\inn{L_\sigma\eta}{\eta}\ge\alpha\norm{\eta}_{X^0}^2\quad 
			\del{\eta\in \ker Q_\sigma }.
		\end{equation}
	\end{lemma}
	\begin{proof}
		In this proof we fix $\s$ and drop the dependence on $\s$ in subscripts.
		All estimates are independent of $\s$.
		
		1. Recall that $L^{(1)}_x$ is the
		planar linearized operator at the simple planar vortex
		$\psi^{(1)}$ as in Sec. \ref{sec:2.1}. The classical stability result  for planar
		vortices states that
		there is $\alpha>0$ such that for every $\eta\in L^2(\omega)$ 
		orthogonal to the symmetry zero modes $G:=i\psi^{(1)},\,T_j:=\pd{\psi^{(1)}}{x_j}$,
		\begin{equation}
			\label{4.7}
			\inn{L^{(1)}_x\eta(\cdot)}{\eta(\cdot)}_{L^2(\omega)}\ge\beta \norm{\eta}_{L^2(\omega)}^2.
		\end{equation}
		This $\beta$ measures the spectral gap at $0$. 
		See \cite[Secs. 7-8]{MR1479248} for a discussion in the same setting.
		Moreover, in \cite{MR1764706} it is shown that $\beta=O(\abs{\log\epsilon}^{-1})$.
		
		Let $L_0:X^k\to X^{k-2}$ be the linearized operator at the lift of $\psi^{(1)}$ to $X^k$.
		Integrating \eqref{4.7} along $z$-direction, and use periodicity to drop the 
		term $\inn{\di_{zz}\eta}{\eta}\ge0$, we find \begin{equation}
			\label{4.7.1}
			\text{$\inn{L_0\eta}{\eta}\ge\beta\norm{\eta}_{X^0}^2$
				if $\eta(\cdot, z)$ is orthogonal to $T_{j}$ and $G$.}
		\end{equation}
		The orthogonality condition for \eqref{4.7.1} holds trivially if $\eta$ is orthogonal to the lift
		of $T_j$ and $G$ to $X^k$. 
		
		
		2. Put $\b{Q}=1- Q$, then we can rewrite \eqref{4.6} as 
		\begin{equation}
			\label{4.8}
			\inn{L\eta}{\eta}=\inn{LQ\eta}{\eta}+\inn{L\b Q\eta}{\eta}.
		\end{equation}
		The first term is $O(\delta^{1/2})\norm{\eta}_{X^2}^2$ by the approximate
		zero mode property \eqref{4.1.2}. The second term further splits
		as $\inn{L\b Q\eta}{\b Q\eta }+\inn{L\b Q\eta}{Q\eta}=\inn{L\b Q\eta}{\b Q\eta }+\inn{\b Q\eta}{LQ\eta}=\inn{L\b Q\eta}{\b Q\eta }+O(\delta^{1/2})\norm{\eta}_{X^2}^2$,
		again by \eqref{4.1.2}. Thus it suffices to control $\inn{L\b Q\eta}{\b Q\eta }$.
		
		Write $\b{\eta}=\b{Q}\eta$. We claim
		\begin{equation}
			\label{4.9}
			\inn{L\b\eta}{\b\eta }\ge \frac{\beta}{4}\norm{\eta}_{X^0}^2,
		\end{equation}
		which, so long as $\delta \ll \beta$, implies \eqref{4.6} with $\al=\beta/8$. 
		
		3. 
		Choose a partition of unity $\chi_1,\,\chi_2$ on $\Omega$,
		such that $\chi_j\ge0,\,\chi_1^2+\chi_2^2=1$ and $\supp \chi_1\subset\Set{(x,z):x-\g(z)<\rho}$
		for $\g=\sbr{\s}_\g$ and some $\rho$ with $\delta<\rho<\diam(\omega)$ to be determined.
		These cut-off functions $\chi_1,\chi_2$  separate $\Omega$ into an inner region, where
		the vortex filament $\psi_\g$ is small, and an outer region, where
		$\abs{\psi_\g}$ is close to $1$.
		
		We use the localization formula \cite[Eqn. (1.1)]{MR676004},
		$$L=\sum \chi_j L \chi_j-\sum \abs{\grad \chi_j}^2.$$
		If we choose $\abs{\grad \chi}\le \rho^{-1}$, then this formula allows us to write
		the l.h.s. of \eqref{4.9} as
		\begin{equation}
			\label{4.10}
			\inn{L\b\eta}{\b\eta }\ge  \inn{L\chi_1\b\eta}{\chi_1\b\eta}+ \inn{L\chi_2\b\eta}{\chi_2\b\eta}-C_1\rho^{-2}\norm{\eta}_{X^0}^2.
		\end{equation}
		Since $\chi_2$ is supported away from the vortex filament, 
		using the far-off asymptotics in \eqref{2.0},
		the second term in the r.h.s.
		can be bounded from below as $\inn{L\chi_2\b\eta}{\chi_2\b\eta}\ge (1-C_2\eps^2\rho^{-2})\norm{\chi_j\eta}_{X^0}^2$.

		Write $L=L_0+V$, where $L_0$ is the linearized operator at the lift of $\psi^{(1)}$,
		and $V$ is defined by this split. Then by the asymptotics \eqref{2.0},
		we have $\norm{V}_{L^\infty}\le C_2\delta\eps^{-1}$, where $C_2$ is independent of $\rho$.
		This implies 
		\begin{equation}
			\label{4.11}
			\inn{L\b\eta}{\b\eta }\ge  \inn{L_0\chi_1\b\eta}{\chi_1\b\eta}+
			(1-C_2\eps^2\rho^{-2})\norm{\chi_j\eta}_{X^0}^2
			-(C_1\rho^{-2}+C_2\delta\eps^{-1})\norm{\eta}_{X^0}^2.
		\end{equation}
		
		Recall $Q$ is the projection onto the approximate zero modes satisfying \eqref{4.1.2}.
		This, the fact that  $\bar\eta\in \ker  Q_\s$, and the lower bound \eqref{4.7.1} 
		together give
		\begin{equation}
			\label{4.12}
			\inn{L_0\chi_1\b\eta}{\chi_1\b\eta}\ge (\beta-C_3(\delta^{1/2}+\rho^{-1})) \norm{\chi_1\b\eta}^2_{X^0}.
		\end{equation}
		Plugging \eqref{4.12} to \eqref{4.11}, and choosing 
		$\rho=2C_3\beta^{-1}$, we find
		\begin{equation}
			\label{4.13}
			\inn{L\b\eta}{\b\eta }\ge \del{\min\del{\frac{\beta}{2}-C_3\delta^{1/2},1-
					4C_2C_3^2\eps^2\beta^2}-(4C_1C_3^2\beta^2+C_2\delta\eps^{-1})}\norm{\eta}_{X^0}^2.
		\end{equation}
		Since $\beta=O(\abs{\log\eps}^{-1})$, for $\delta=\eps^{1+s},\,s>0$ and all $\eps\ll1$,
		we get the claim \eqref{4.9} from \eqref{4.13}.
	\end{proof}

	\section*{Acknowledgments}
	
	The Author is supported by Danish National Research Foundation grant
	CPH-GEOTOP-DNRF151. 
	
	\section*{Declarations}

	\begin{itemize}
		\item Competing interests: The Author has no conflicts of interest to declare that are relevant to the content of this article.
	\end{itemize}

	\appendix
		
		\section{Properties of the Fr\'echet derivatives}
		In this section we consider the Fr\'echet derivatives of the immersion $f$ defined in \eqref{2.1}. 
		
		\subsection{Basic variational calculus}
		First,  we recall some basic elements of variational derivatives that we used 
		repeatedly.
		For details, see for instance \cite[Appendix C]{MR2431434}
		
		\emph{Fr\'echet derivative}. Let $X,\,Y$ be two Banach spaces, $U$ be an open set in $X$.
		For a map $g:U\subset X\to Y$ and a vector $u\in U$,
		the Fr\'echet derivative $dg(u)$ is a linear map from $X\to Y$ such that
		$g(u+v)-g(u)-dg(u)v=o(\norm{v}_X)$ for every $v\in X$. 
		If $dg(u)$ exists at $u$, then it is unique. If $dg(u)$ exists for every $u\in U$,
		and the map $u\mapsto dg(u)$ is continuous from $U$ to the space of linear operators $L(X,Y)$, then we
		we say $g$ is $C^1$ on $U$. 
		In this case, $dg(u)$ is uniquely given by 
		$$v\mapsto \od{g(u+tv)}{t}\vert_{t=0}\quad (v\in X).$$
		Iteratively, we can define higher order derivatives this way.

		\emph{Gradient and Hessian}. If $X$ is a Hilbert space over scalar field $Y$, then by Riesz representation,
		we can identify $dg(u)$ as an element in $X$, denoted $g'(u)$. The vector
		$g'(u)$ is called the $X$-gradient of $g$. 
		Similarly, we denote $g''(u)$ the second-order Fr\'echet derivative
		$d^2g(u)$. If $g$ is $C^2$, then $g''$ can be identified as a symmetric 
		linear operator
		uniquely determined  by the relation  
		$$\inn{g''(u)v}{w}=\md{g(u+tv+sw)}{2}{t}{}{s}{}\vert_{s=t=0}\quad (v,w\in X).$$
		
		\emph{Expansion}. Let $X$ be a Hilbert space over scalar field $Y$. Suppose $g$ is $C^2$ on $U\subset X$. Define a scalar function $\phi(t):=g(u+tv)$ for vectors
		$v,w$ such that $v+tw\in U$ for every $0\le t\le 1$.
		Then the elementary Taylor expansion at $\phi(1)$ gives 
		$$g(v+w)=g(v)+\inn{g'(v)}{w}+\frac{1}{2}\inn{g''(v)w}{w}+o(\norm{w}^2).$$
		Here we have used the definition of $g'$ and $g''$ from the last paragraph.
		
		\emph{Composition}. Let $\Omega\subset \Rr^d$ be a bounded domain with smooth boundary.
		Fix $r>d/2,\,f\in C^{r+1}(\Rr^n)$. For $u:\Omega\to \Rr^n$,
		define a map $g:u\mapsto f\circ u$. Then $g:H^r(\Omega)\to H^r(\Omega)$,
		and is $C^1$. 
		The Fr\'echet derivative is given by $v\mapsto \grad f\cdot v$.
		
		\subsection{Various uniform estimates}
		In this section we assume $\epsilon\ll1$ in \eqref{1.1}. 
		For two complex numbers, we use the real inner product $\inn{u}{v}=\Re (\bar{u}v)$.
		
		Fix some $\alpha>0$. 
		Using the definition of the immersion $f$ in Section \ref{sec:2.2}, for $\sigma=(\lambda,\gamma)\in \Sigma_{\epsilon^\alpha}$ and $(\mu,\xi)\in Y^k$, we compute the Fre\'echet derivative of $f$ as
		\begin{equation}
			\label{A.1}
			df(\sigma)(\mu,\xi)=e^{i\lambda}(i\mu\psi_\g-\grad_x\psi_\g\cdot \xi).
		\end{equation}
		This is uniformly bounded in $\sigma$ as an operator from $Y^0\to X^0$, since
		$$\begin{aligned}
			\norm{df(\sigma)(\mu,\xi)}_{X^0}&\le
			\mu\norm{\psi_\g}_{X^0}+\norm{\grad_x\psi_\g\cdot\xi}_{X^0}\\
			&\le
			\mu\del{\norm{\psi^{(1)}}_{L^2(\omega)}+O(\epsilon^{2+\alpha/2})}\\
			&+\del{\norm{\grad\psi^{(1)}}_{L^2(\omega)}+O(\epsilon^{1+\alpha/2})}\norm{\xi}_{C^0}\\&\le C(\Omega)\abs{\log\epsilon}^{1/2}\norm{\sigma}_{Y^0}.
		\end{aligned}$$
		Here we have used \eqref{2.0.5} and \eqref{4.3}.
		Using \eqref{A.1} and \eqref{3.3.00}, one can  get a uniform estimate for $d_\sigma df(\sigma):Y^0\to L(Y^0, X^0).$ 
		
		Write an element in $Y^k$ as $\sigma=([\sigma]_\lambda,[\sigma]_\xi)$. 
		The adjoint operator $df(\sigma)^*$ is determined by the relation
		$$\begin{aligned}
			&\inn{ df(\sigma)(\mu,\xi)}{\phi}\\
			&=\int_x\int_z\inn{e^{i\lambda}(i\mu\psi_\g-\grad_x\psi_\g)\cdot \xi}{\phi}\\&=\mu[df(\sigma)^*\phi]_\lambda +\int_z\xi\cdot [df(\sigma)^*\phi]_\gamma\quad (\phi\in X^0).
		\end{aligned}$$
		Here and in the remaining of this section, it is understood that various integrals are
		taken over $(x,z)\in\omega\times I=\Omega$.
		
		By Fubini's Theorem and the identity $\inn{v\cdot w}{ u}=\inn{u}{v}\cdot w$ for real vector $w$, the above relation implies
		\begin{equation}
			\label{A.2}
			df(\sigma)^*\phi=\del{\int_x\int_z\inn{\phi}{ie^{i\lambda}\psi_\g},-\int_x\inn{{e^{i\lambda}\grad_x\psi_\g(x,\cdot)}}{\phi(x,\cdot)}}.
		\end{equation}
		This adjoint operator  is also uniformly bounded in $\sigma$ with $\norm{df(\sigma)^*}_{X^0\to Y^0}
		\le C\abs{\log\epsilon}^{1/2}.$ Moreover, by Sobolev embedding,  $df(\sigma)^*$ maps $X^2$ into $Y^0$.  
		Using \eqref{A.2} and \eqref{3.3.00}, one can  get a uniform estimate for $d_\sigma df(\sigma)^*:Y^0\to L(X^0, Y^0).$ 
		
		The operator $\J_\sigma$ is defined in \eqref{2.3}.
		It induces a symplectic form w.r.t. the inner product
		\eqref{2.00Y} on the tangent bundle $T\Sigma$, since
		$$\inn{\J_\sigma\chi}{\chi}=\inn{g_\sigma^*J^{-1}g_\sigma\chi}{\chi}=\inn{J^{-1}g_\sigma\chi}{g_\sigma\chi}=0\quad (\chi\in Y^0).$$
		Using \eqref{A.1}-\eqref{A.2}, we can compute $\J_\sigma$ explicitly as
		\begin{equation}
			\label{A.3}
			\begin{split}
				&[\J_\sigma(\mu,\xi)]_\lambda =-\int_x\int_z\inn{\grad_x\psi_\g\cdot\xi}{\psi_\g},\\
				&[\J_\sigma(\mu,\xi)]_\g =\mu\int_x \inn{\grad_x\psi_\g(x,\cdot)}{\psi_\g(x,\cdot)}+\int_x \inn{\grad_x\psi_\g(x,\cdot)\cdot J\xi(\cdot)}{\grad_x\psi_\g(x,\cdot)}.
			\end{split}
		\end{equation}
		Here we have used the identity that for any complex-valued $C^1$ function $\phi$ and vector field $\chi$ in $\Rr^2$,
		by the Cauchy-Riemann equation, 
		$$-i\grad\phi\cdot \chi=\grad{\phi}\cdot J\chi.$$
		
		One can also write $\J_\sigma=\J_\sigma(\mu,\xi(z))$ as the multiplication operator by the matrix
		$B_{ij}$, where
		\begin{equation}
			\label{A.4}
			\begin{aligned}
				&B_{ii}=0\quad(i=1,2,3),\\
				&B_{12}=-\int_x\int_z\inn{\pd{\psi_\g}{x_1}(x,z)}{\psi_\g(x,z)},\\
				&B_{13}=-\int_x\int_z\inn{\pd{\psi_\g}{x_2}(x,z)}{\psi_\g(x,z)},\\
				&B_{21}=\int_x\inn{\pd{\psi_\g}{x_1}(x,z)}{\psi_\g(x,z)},\\
				&B_{23}=\int_x\inn{\pd{\psi_\g}{x_1}(x,z)}{\pd{\psi_\g}{x_1}(x,z)},\\
				&B_{31}=\int_x\inn{\pd{\psi_\g}{x_2}(x,z)}{\psi_\g(x,z)},\\
				&B_{32}=-\int_x\inn{\pd{\psi_\g}{x_2}(x,z)}{\pd{\psi_\g}{x_2}(x,z)}.
			\end{aligned}
		\end{equation}
		Here we have used the Cauchy-Riemann equation to eliminate certain cross-terms of the form $\inn{\pd{\psi_\g}{x_1}}{\pd{\psi_\g}{x_2}}$.
		
		Using \eqref{A.4}, the fact $\norm{\grad\psi^{(1)}}_{L^2(\omega)}\sim C(\omega)\abs{\log\epsilon}^{1/2}$ (see \eqref{2.0.5} and 
		\cite[Chap. V.1]{MR1269538}), and the asymptotics \eqref{2.0} for $\psi^{(1)}$,
		one can check that $\J_\sigma$ is invertible and satisfies  the uniform
		estimates
		\begin{align}	
			&\norm{\J_\sigma}_{Y^k\to Y^k}\le C(\omega)\abs{\log\epsilon}^2\label{A.5},\\
			&\norm{\J_\sigma^{-1}}_{Y^k\to Y^k}\le C(\omega)\abs{\log\epsilon}^{-2}\label{A.6}
		\end{align} 
		for every $k\in\N$.

	\bibliography{bibfile}

\end{document}